\documentclass{article}  % list options between brackets
\usepackage{amssymb}  % list packages between braces
\usepackage{amsthm}
\usepackage{eucal}
\usepackage{verbatim}
\usepackage[dvips]{graphicx}
\usepackage{multirow}
\usepackage{fancyhdr}
\usepackage{color}
\usepackage{enumerate}
\usepackage{calrsfs}
\usepackage{fullpage}
\usepackage{hyperref}
\usepackage{amsmath}
\DeclareMathAlphabet{\pazocal}{OMS}{zplm}{m}{n}
\DeclareMathAlphabet{\mathcalligra}{T1}{calligra}{m}{n}

\newtheorem{theorem}{Theorem}
\newtheorem{lemma}{Lemma}
\newtheorem{definition}{Definition}
\newtheorem{remark}{Remark}
\newtheorem{proposition}{Proposition}
\newtheorem{corollary}{Corollary}

\makeatletter
\newcommand{\leqnomode}{\tagsleft@true}
\newcommand{\reqnomode}{\tagsleft@false}
\makeatother
\def\({\begin{eqnarray}}
\def\){\end{eqnarray}}
\def\[{\begin{eqnarray*}}
\def\]{\end{eqnarray*}}
\def\part#1#2{\frac{\partial #1}{\partial #2}}

\title{A simple consensus model for an increasing population of agents with i.i.d incoming opinions.}
\author{Ioannis Markou}

\begin{document}

\maketitle

\begin{abstract}
In this short note we study what happens in a symmetric opinion model when we send the total interacting population $N(t)$ to infinity as $t \to \infty$. We assume that new population enters the system with opinions that are i.i.d random vectors with finite mean and variance. We give sharp conditions on the rate of population growth that is required for convergence to a global consensus in opinions. More particularly, we show that if the total population increases at a rate $N(t)\sim e^{t^\alpha}$, then $\alpha<1$ is necessary and sufficient condition for convergence to the mean of incoming opinions, and the convergence is achieved at an algebraic rate.
\end{abstract}



\section{Introduction}
\label{sec:Introduction}

In the past few decades the interest in the study of social dynamics and more particularly in opinion modeling has re-emerged. Opinion models are first order models where the unknown variables represent opinions for a fixed number of agents that change dynamically according to some rule. Depending on the network structure, each opinion may be affected by either the neighboring opinions i.e. opinions that are not too distant (see \cite{HegKraus, Kraus}), or by all other opinions with some appropriate weight in networks with all-to-all interaction topology. For the insterested reader we suggest the review articles \cite{AyCa, CaFoLo}, and the excellent bibliographic survey in \cite{MasVelNal} with a list of some of the most important contributions.

The purpose of this work is to study the role of population growth in a simple symmetric opinion model with an all-to-all interaction topology. We study how sending the number of agents to infinity 	affects the long time asymptotic behavior of the system. In particular, we find sharp conditions for the emergence of consensus, when we keep adding agents in the system with incoming opinions that are independent and identically distributed (i.i.d). We should point out that the study of population growth in consensus dynamics and collective behavior more generally is a novel idea \cite{Mar}.

\section{Preliminaries and Main results}
\label{Preliminaries}

We start by assuming an increasing population $N(t):[0,\infty) \to \mathbb{N}$ that counts the number of agents at time $t$, which for the sake of this work we consider to be unbounded so that $N(t) \xrightarrow{t \to \infty}\infty$. We also assume an increasing, unbounded sequence of times $\{ t_j\}_{j=0}^\infty$, with $t_0=0$ s.t. $N(t)$ increases by a unit at each $t_j$ starting from $N(t_0)=N_0$. In short, the population is given by
\begin{equation} \label{eq:population} N(t):=N_0+ \sum_{j=0}^\infty j \chi_{[t_j,t_{j+1})}(t) , \end{equation}
where $\chi_A(t)$ is the characteristic function on a set $A \subset \mathbb{R}$.

In our model we study the evolution in time of an opinion vector $x_i(t) \in \mathbb{R}^d$ for agent $i$, where the index $i$ takes values over the existing population $N(t)$. For each time interval $[t_k,t_{k+1})$ in which the population $N(t)$ remains fixed, every opinion is updated continuously in time by averaging all the existing opinions of other agents with respect to a symmetric weight $\psi(|x_i-x_j|)$. The interaction weight $\psi(|x_i-x_j|)$ is a function of the euclidean distance $|\cdot|$ in $\mathbb{R}^d$ between the opinions of agents $i$, $j$. For the purpose of this paper the weight function $\psi(r)$ is only assumed to be a Lipschitz function $\psi(r):\mathbb{R}_+ \to \mathbb{R}_+ $ which is bounded below and above, i.e. $\psi_* \leq \psi(r)\leq \psi_M$. Hence, we end up with the system
\begin{equation} \label{eq:main} \dot{x}_i(t)=\frac{1}{N(t)} \sum \limits_{j=1}^{N(t)}\psi(|x_j(t)-x_i(t)|)(x_j(t)-x_i(t)) ,\quad 1 \leq i \leq N(t), \quad t \geq 0 .\end{equation}
The initial conditions of the model for $t_0=0$ are $x_j^0$ for $1 \leq j \leq N_0$. At each time step $t_k$ for $k \geq 1$, agent $N(t_k)$ is introduced with an opinion that is a function $x_{N(t_k)}:[t_k,\infty) \to \mathbb{R}^d$ having initial condition $X_k$. Overall, the initial conditions are
\begin{equation} \label{eq:IC}x_j(t_0)=x_j^0 \quad \text{for} \quad 1 \leq j \leq N_0, \quad \text{and}\quad x_{N(t_k)}(t_k)=X_k , \quad k \geq 1.\end{equation}
\begin{comment}
This implies that for the population that existed in the previous time, the initial conditions are a direct continuation of their previous values, where $x_i(t_k^-)$ stands for the left limit of $x_i(t)$ at $t_k$. The value $X_k$ is the opinion of the added agent at $t_k$.
\end{comment}

We assume that the incoming opinions $\{X_k \}_{k \geq 1}$ are an infinite sequence of random vectors which are i.i.d with finite mean $\mathbb{E}(X_k)=m \in \mathbb{R}^d$, and Variance $\mathrm{Var}(X_k)=\sigma^2>0$. Here we define the notion of variance $\mathrm{Var}(X_k)$ for the random vectors $X_k$ as the trace of the covariance matrix $\mathrm{Cov}(X_k)=\mathbb{E}((X_k-\mathbb{E}(X_k))(X_k-\mathbb{E}(X_k))^T)$ i.e. $\mathrm{Var}(X_k)=tr(\mathrm{Cov}(X_k))=\mathbb{E}(|X_k -m|^2)=\sigma^2$.
Next, we give the definition of the notion of convergence we use.
\begin{definition} Let $N(t)$ a population that evolves according to \eqref{eq:population}, and assume a classical solution to \eqref{eq:main}-\eqref{eq:IC}. We say that the sequence converges in the \textbf{mean variance} sense to the average $m$ iff \begin{equation*} \mathbb{E}\left( \frac{1}{N(t)}\sum_{j=1}^{N(t)}|x_j(t)-m|^2\right) \xrightarrow{t \to \infty} 0 . \end{equation*}
The expectation $\mathbb{E}(\cdot)$ is the expectation of the stochastic process $\{x_j(t)\}_{j=1}^{N(t)}$.
\end{definition}
In the following result we give sharp conditions on the rate of $N(t)$ growth, i.e. the distribution of the $t_k$ values at infinity so that we have convergence in the sense defined above.
\begin{theorem} \label{theorem:Main} We consider an increasing population $N(t):[0,\infty) \to \mathbb{N}$ that follows \eqref{eq:population}, for some increasing, unbounded, sequence of times $\{ t_j\}_{j=1}^\infty$. We also assume a classical solution of \eqref{eq:main} - \eqref{eq:IC} and that the incoming opinions $X_k$ are i.i.d. random vectors with finite mean $\mathbb{E}(X_k)=m$ and variance $\mathrm{Var}(X_k)=\sigma^2$. Then we have \textbf{mean variance} convergence of the opinions to the average opinion of incoming agents $m$ if condition
\begin{equation}  \label{eq:condition1}\sum_{k=1}^n \frac{1}{k} e^{-\psi_* (t_n-t_k)} \xrightarrow{n \to \infty} 0  \tag{C1} \end{equation} holds. Moreover, mean variance convergence fails if
\begin{equation} \label{eq:condition2}\limsup \limits_{n \to \infty} \sum_{k=1}^n \frac{1}{k} e^{-\psi_M (t_n-t_k)} >0 . \tag{C2}\end{equation}
\end{theorem}

In the result that follows we show that for every sub-exponential population growth we do have average mean square convergence to $m$, wheareas we don't in the case of exponential or faster.

\begin{theorem} We assume the same assumptions as in Theorem \ref{theorem:Main}. We also assume that the population $N(t)$ grows at a general exponential rate, i.e $N(t)= \lfloor e^{t^\alpha} \rfloor$, for some $\alpha>0$. Then we have convergence in the average mean square sense iff the growth is sub-exponential, i.e. for $\alpha<1$. Furthermore, for any $\beta<1-\alpha$ there exists some $T_0=T_0(a,\psi_*,\beta^*)>0$ s.t we have the convergence rate
\begin{equation} \mathbb{E}\left( \frac{1}{N(t)}\sum_{j=1}^{N(t)}|x_j(t)-m|^2\right) < O \left(\frac{1}{t^{\beta}} \right), \qquad t \geq T_0 . \end{equation}
\end{theorem}

\section{Moments and Variance estimates}
\label{sec:Estimates}

\subsection{Moments of the model}
We define the first moments of our model at any time $t$ by taking the average over population $N(t)$, i.e.
\begin{equation} \label{eq:moments} m_0(t):=\frac{1}{N(t)}\sum \limits_{j=1}^{N(t)}1 , \quad m_1(t):=\frac{1}{N(t)}\sum \limits_{j=1}^{N(t)} x_j(t), \quad   m_2(t):=\frac{1}{N(t)}\sum \limits_{j=1}^{N(t)} |x_j(t)|^2 .\end{equation}
We know that the symmetry in the interactions between two agents $i$ and $j$ implies \begin{equation*} \dot{m}_1(t)=0 , \qquad \dot{m}_2(t)=D(t), \qquad t \in (t_k, t_{k+1}), \end{equation*}
with $D(t)$ being the dissipation defined by \begin{equation*} D(t)=-\frac{1}{(N_0+k)^2}\sum_{i=1}^{N_0+k}\sum_{j=1}^{N_0+k} \psi(|x_j(t)-x_i(t)|)|x_j(t)-x_i(t)|^2 , \quad t \in (t_k, t_{k+1}).\end{equation*}

As a new agent is introduced every $t_k$ time, the macroscopic moments $m_1(t)$ and $m_2(t)$ become discontinuous at these points. In order to study these discontinuities we have to introduce the notation for the jump of a piecewise continuous function $h(t)$ at $t_k$. This is defined by $\Delta_j h(t_k):=h(t_k^+)-h(t_k^-)$, with $h(t_k^+)$ and $h(t_k^-)$ being the the right hand and left hand limits of $h$ at $t_k$. More particularly, when $h(t)$ is continuous from the right, then $\Delta_j h(t_k):=h(t_k)-h(t_k^-)$.
Now we can show
\begin{proposition} \label{thm:prop1}
Assume a population $N(t)$ that increases according to \eqref{eq:population} and a classical solution of \eqref{eq:main} - \eqref{eq:IC}. The moments defined in \eqref{eq:moments} satisfy for any $t \in [t_k, t_{k+1})$
\begin{equation*} \begin{aligned} (i) \quad m_0(t)&=1, \qquad (ii) \quad m_1(t)=\frac{1}{N_0+k}\left( \sum_{j=1}^{N_0}x_j^0 +\sum_{j=1}^k X_j\right) ,  \\  (iii) \quad m_2(t)&= \frac{1}{N_0+k}\left( \sum_{j=1}^{N_0}|x_j^0|^2+\sum_{j=1}^k |X_k|^2\right) +\sum_{j=0}^{k-1}\frac{N_0+j}{N_0+k} \int_{t_j}^{t_{j+1}} D(s) \, ds +\int_{t_k}^t D(s) \, ds .  \end{aligned} \end{equation*}
\end{proposition}

\begin{proof}
We have that $m_0(t)=1$ for all $t \geq 0$ by definition. Now in terms of $m_1(t)$, this is a moment that is not presrved unlike the symmetric fixed population model.
The jump of $m_1(t)$ at $t_k$ is
$ \Delta_j m_1(t_k)=m_1(t_k)-m_1(t_k^-)$.
We can compute the value of $\Delta_j m_1(t_k)$ using the property $x_j(t_k)=x_j(t_k^-)$ for  $j=1,\ldots,N_0+k-1$, i.e.
\begin{equation} \label{eq:m_1_jump_formula}\begin{aligned} \Delta_j m_1(t_k)&=\frac{1}{N_0+k}\sum_{j=1}^{N_0+k}x_j(t_k)-m_1(t_k^-)=\frac{1}{N_0+k} \left(\sum_{j=1}^{N_0+k-1}x_j(t_k^-)+X_k \right) -m_1(t_k^-) \\ &=\frac{1}{N_0+k}\left( (N_0+k-1)m_1(t_k^-)+X_k\right)-m_1(t_k^-)=\frac{1}{N_0+k}X_k-\frac{1}{N_0+k}m_1(t_k^-) .\end{aligned} \end{equation}
We have that $m_1(t_k^-)=m_1(t_{k-1})$ (since $\dot{m}_1(t)=0$ on $(t_{k-1},t_k)$). Then using \eqref{eq:m_1_jump_formula} we get
\begin{equation*}  m_1(t_k)=m_1(t_k^-)+\Delta_j m_1(t_k)=\frac{1}{N_0+k} X_k +\left(1-\frac{1}{N_0+k} \right)m_1(t_{k-1}), \end{equation*}
which is the induction step to prove the expression for $m_1(t)$.

The formula for the jump of the second moment $\Delta_j m_2(t_k)$ is shown like in \eqref{eq:m_1_jump_formula} to be
\begin{equation} \label{eq:m_2_jump_formula} \Delta_j m_2(t_k)=\frac{1}{N_0+k}|X_k|^2 -\frac{1}{N_0+k}m_2(t_k^-) .\end{equation}
Again, with the use of induction we prove the equation for $m_2(t)$. If we show it is true for $m_2(t_k)$, for every $k \geq 0$, then the $m_2(t)$ expression follows after the fact that $\dot{m}_2(t)=D(t)$ on $(t_n,t_{n+1})$. Indeed, for $k=0$ it holds trivially. Assume it holds also for $k-1$. Then, with the help of \eqref{eq:m_2_jump_formula} we get
\begin{align*} m_2(t_k)&=m_2(t_k^-)+\Delta_j m_2(t_k)= \Delta_j m_2(t_k)+\int_{t_{k-1}}^{t_k} D(s)\, ds +m_2(t_{k-1}) \\ &=\frac{1}{N_0+k}|X_k|^2 -\frac{1}{N_0+k}m_2(t_k^-)+ \int_{t_{k-1}}^{t_k} D(s)\, ds +m_2(t_{k-1}) \\ &= \frac{1}{N_0+k}|X_k|^2 +\left( 1-\frac{1}{N_0+k}\right)\left( \int_{t_{k-1}}^{t_k} D(s)\, ds +m_2(t_{k-1}) \right), \end{align*}
which is the formula for the induction step of $m_2(t_k)$.
\end{proof}
A direct result of Proposition \ref{thm:prop1} are the following formulas
\begin{corollary} Based on the moment expressions we obtained in Proposition \ref{thm:prop1}, for any $t \in [t_k,t_{k+1})$, we have \begin{align*} \mathbb{E}(|m_1(t)-m|^2)&=\frac{A+k \sigma^2}{(N_0+k)^2}\quad \text{where} \quad A= \left|\sum \limits_{j=1}^{N_0}(x_j^0 -m) \right|^2 , \\ \mathbb{E}(|m_1(t)|^2)&=\frac{B+ k \sigma^2 +2k C +k^2 |m|^2}{(N_0+k)^2} , \qquad B=\left|\sum \limits_{j=1}^{N_0} x_j^0  \right|^2 ,\quad  C= \sum \limits_{j=1}^{N_0} x_j^0 \cdot m .\end{align*}
This implies in particular that $\mathbb{E}(|m_1(t)-m|^2)=O(\frac{1}{k})=O \left(\frac{1}{N(t)}\right)$ as $k \to \infty$.
\end{corollary}

\subsection{The Variance functional}
We introduce the variance of the model which we define by
\begin{equation} \label{eq:Variance} \pazocal{V}(\{x(t)\}):=\frac{1}{N(t)}\sum_{j=1}^{N(t)}|x_j(t)-m_1(t)|^2, \quad  t \geq 0.\end{equation}
The following result gives the main estimates that we use in our proof
\begin{proposition} \label{thm:prop2} The expectation of the variance functional \eqref{eq:Variance} is a piecewise $C^1$ function s.t. \begin{equation} \label{eq:Variance_dissipation} \frac{d}{dt} \mathbb{E}(\pazocal{V}(\{x(t)\})) \leq -\psi_* \mathbb{E}(\pazocal{V}(\{x(t)\})), \qquad t \in (t_k, t_{k+1}), \quad k \geq 0 . \end{equation}
The following bounds for the jump discontinuities of $\mathbb{E}(\pazocal{V}(\{x(t)\}))$ at $t_k$ hold :
\begin{equation} \label{eq:Upper_Lower_bound}\begin{aligned}\Delta_j \mathbb{E}(\pazocal{V}(\{ x(t_k)\})) &\leq \frac{1}{N_0+k} \left( c_k \sigma^2- \mathbb{E}(\pazocal{V}(\{ x(t_k^-)\}))\right) +\frac{C}{(N_0+k)^2} , \\  \Delta_j \mathbb{E}(\pazocal{V}(\{ x(t_k)\})) & \geq \frac{1}{N_0+k} \left( c_k \sigma^2- \mathbb{E}(\pazocal{V}(\{ x(t_k^-)\}))\right) , \end{aligned} \end{equation}
for $c_k=\frac{(k+2N_0)(k-1)}{(N_0+k)^2}\xrightarrow{k \to \infty}1$, and $C=C(N_0, x^0, m,\sigma)$ is independent of $k$.
\end{proposition}

\begin{proof}
We can write $\pazocal{V}(\{x(t)\})$ alternatively as $\pazocal{V}(\{x(t)\})=m_2(t)-m_1(t)^2$, so after differentiating in time we have $\dot{\pazocal{V}}(\{x(t)\})=D(t)$ on every $(t_k, t_{k+1})$ interval. Taking the expectation of this and the lower bound on the influence weight we get $ \mathbb{E}(\dot{\pazocal{V}}(\{x(t)\}))\leq -\psi_* \mathbb{E}(\pazocal{V}(\{x(t)\}))$. The last step, i.e. commuting the derivative with the integral to get \eqref{eq:Variance_dissipation} is a simple exercise in analysis.

The jump of $\pazocal{V}(\{x(t)\})$ at $t_k$, defined by $\Delta_j \pazocal{V}(\{x(t_k)\})=\pazocal{V}(\{x(t_k)\}) -\pazocal{V}(\{x(t_k^-)\})$ is
\begin{equation*}\begin{aligned} \Delta_j \pazocal{V}(\{x(t_k)\}) &=\frac{1}{N_0+k}\sum_{j=1}^{N_0+k} |x_j(t_k)-m_1(t_k)|^2
- \frac{1}{N_0+k-1} \sum_{j=1}^{N_0+k-1} |x_j(t_k^-)-m_1(t_k^-)|^2
\\ &=\frac{1}{N_0+k} |X_k-m_1(t_k)|^2 + \frac{1}{N_0+k} \sum_{j=1}^{N_0+k-1} |x_j(t_k)-m_1(t_k)|^2
\\ &-\frac{1}{N_0+k-1} \sum_{j=1}^{N_0+k-1} |x_j(t_k^-)-m_1(t_k^-)|^2 = \frac{1}{N_0+k} |X_k-m_1(t_k)|^2
\\ &+\frac{1}{N_0+k} \sum_{j=1}^{N_0+ k-1} |x_j(t_k)-m_1(t_k^-)-\Delta_j m_1(t_k)|^2 - \frac{1}{N_0+k-1} \sum_{j=1}^{N_0+k-1} |x_j(t_k^-)-m_1(t_k^-)|^2 \end{aligned} \end{equation*}
We treat the second term as
\begin{equation*} \frac{1}{N_0+k} \sum_{j=1}^{N_0+ k-1} |x_j(t_k)-m_1(t_k^-)-\Delta_j m_1(t_k)|^2 = \frac{1}{N_0+k}\sum_{j=1}^{N_0+ k-1} |x_j(t_k^-)-m_1(t_k^-)|^2 +\frac{N_0+k-1}{N_0+k}|\Delta_j m_1(t_k)|^2 .\end{equation*}
Hence, with the help of the definition of $m_1(t)$ and the formula \eqref{eq:m_1_jump_formula} for $\Delta_j m_1(t_k)$ we get
\begin{equation} \label{eq:DeltaV} \begin{aligned}
\Delta_j \pazocal{V}(\{x(t_k)\})&=\frac{1}{N_0+k} |X_k-m_1(t_k)|^2 +\frac{N_0+k-1}{(N_0+k)^3} |X_k -m_1(t_k^-)|^2 \\ &-\frac{1}{(N_0+k)(N_0+k-1)}\sum_{j=1}^{N_0+k-1} |x_j(t_k^-)-m_1(t_k^-)|^2 \\ &=\frac{1}{N_0+k} |X_k-m_1(t_k)|^2 +\frac{N_0+k-1}{(N_0+k)^3} |X_k -m_1(t_k^-)|^2 -\frac{1}{N_0+k}\pazocal{V}(\{x(t_k^-)\}) .
 \end{aligned}\end{equation}
Let's us first compute the term $|X_k-m_1(t_k)|^2$. We start by writing $ m_1(t_k)-X_k$ as \begin{equation} \label{eq:m_1-X_k} \begin{aligned} m_1(t_k)-X_k &=\frac{1}{N_0+k}\left( \sum_{j=1}^{N_0}x_j^0 +\sum_{j=1}^{k-1}X_j \right)+\frac{1}{N_0+k}X_k -X_k \\ &=\frac{1}{N_0+k}\left( \sum_{j=1}^{N_0}x_j^0 +\sum_{j=1}^{k-1}(X_j-X_k) -N_0 X_k \right) .\end{aligned}\end{equation}
Taking the square of \eqref{eq:m_1-X_k} yields
\begin{equation} \label{eq:|m_1-X_k|^2} \begin{aligned} |m_1(t_k)-X_k|^2&=\frac{1}{(N_0+k)^2}\left|\sum_{j=1}^{k-1}(X_j-X_k)\right|^2 + \frac{1}{(N_0+k)^2} \left|\sum_{j=1}^{N_0}x_j^0 -N_0 X_k \right|^2 \\ &+\frac{2}{(N_0+k)^2} \left( \sum_{j=1}^{N_0}x_j^0 -N_0 X_k \right) \cdot \left( \sum_{j=1}^{k-1}(X_j-X_k )\right) . \end{aligned} \end{equation}
We now compute the expectation of $\mathbb{E}(|m_1(t_k)-X_k|^2)$, keeping in mind that we want to bound it from below  as well. First, the expectation of $\left|\sum \limits_{j=1}^{k-1}(X_j-X_k)\right|^2$ is
\begin{equation} \label{eq:|X_j-X_k|^2}\begin{aligned}\mathbb{E} &\left( \left|\sum_{j=1}^{k-1} (X_j-X_k) \right|^2 \right) =\mathbb{E} \left(\sum_{i=1}^{k-1}(X_i-X_k) \cdot \sum_{j=1}^{k-1}(X_j-X_k)\right) =\sum_{i,j=1}^{k-1} \mathbb{E}((X_i-X_k)\cdot (X_j-X_k)) \\ &=\sum_{i,j=1}^{k-1}\mathbb{E}\left(X_i \cdot X_j -(X_i+X_j)\cdot X_k +|X_k|^2 \right) =k(k-1) \sigma^2 . \end{aligned} \end{equation}
We compute the expectation of the last term in \eqref{eq:|m_1-X_k|^2} directly, instead of controling it by some inequality e.g. Cauchy-Schwartz and Young. In fact we see that this term is of order $O(k)$, instead of $O(k^2)$,
\begin{equation} \label{eq:cross_term}\begin{aligned}&\mathbb{E} \left( \left( \sum_{j=1}^{N_0}x_j^0 -N_0 X_k \right) \cdot \left( \sum_{j=1}^{k-1}(X_j-X_k )\right)\right)=-N_0 \mathbb{E}\left(X_k \cdot \sum_{j=1}^{k-1} (X_j -X_k) \right) \\ &=-N_0 \left( \sum_{j=1}^{k-1} \mathbb{E}(X_k) \cdot \mathbb{E}(X_j) -\sum_{j=1}^{k-1}\mathbb{E}(|X_k|^2)\right)=N_0 (k-1) \sigma^2.\end{aligned}
\end{equation}
Considering also that the expectation of the second term in \eqref{eq:|m_1-X_k|^2} is bounded by $C_1/(N_0+k)^2$, for some $C_1(N_0, x^0, m,\sigma)>0$, and using \eqref{eq:|m_1-X_k|^2}-\eqref{eq:|X_j-X_k|^2}-\eqref{eq:cross_term} we get the upper bound
\begin{equation} \label{eq:Ineq1} \frac{1}{N_0+k}\mathbb{E}(|m_1(t_k)-X_k|^2) \leq \frac{(k+2N_0)(k-1)}{(N_0 +k)^3} \sigma^2+\frac{C_1}{(N_0+k)^3} .\end{equation}
We also have the strict lower bound (due to the positive middle term in \eqref{eq:|m_1-X_k|^2}) \begin{equation}  \label{eq:Ineq2} \frac{1}{N_0+k}\mathbb{E}(|m_1(t_k)-X_k|^2) \geq \frac{(k+2N_0)(k-1)}{(N_0 +k)^3} \sigma^2. \end{equation}
The expectation of the second term in \eqref{eq:DeltaV} has the following bound  for $C_2(N_0, x^0, m,\sigma)>0$
\begin{equation} \label{eq:Ineq3} \frac{N_0+k-1}{(N_0+k)^3}\mathbb{E}(|X_k-m_1(t_k^-)|^2)\leq \frac{C_2}{(N_0+k)^2} .\end{equation}
Finally, taking the expectation in \eqref{eq:DeltaV} and using \eqref{eq:Ineq1}-\eqref{eq:Ineq2}-\eqref{eq:Ineq3}  we get \eqref{eq:Upper_Lower_bound}, but with $\mathbb{E}(\Delta_j \pazocal{V}(\{x(t_k)\}))$ in the place of $\Delta_j \mathbb{E}(\pazocal{V}(\{ x(t_k)\}))$. These two expressions can be shown equal with the help of dominated convergence theorem and the regularity of $\pazocal{V}(\{ x(t_k)\})$ on $(t_n,t_{n+1})$.

\end{proof}

\begin{comment}
\begin{remark} We observe that $\mathbb{E}(\Delta_j \pazocal{V}(\{ x(t_k)\}))$ vanishes as $k \to \infty$. Also, if $\mathrm{Var}(X_k)/N_0+k$ is summable, i.e. $\sum_{k \geq 0}\frac{\mathrm{Var}(X_k)}{N_0+k}<\infty$ then we have $\sum_{k \geq 0}\mathbb{E}(\Delta_j \pazocal{V}(\{ x(t_k)\}))<\infty$. \end{remark}
\end{comment}

\section{Proof of Theorems}
\label{sec:Proofs}
In the lemma that follows we give control for a piecewise $C^1$ function that has a discrete number of jump discontinuities and also decreases at exponential rate everywhere else.
\begin{lemma} \label{lem:Main_lemma} Assume a nonnegative function $y(t):[0,\infty) \to \mathbb{R}_+$ with $y(0)=y_0$, which is piecewise $C^1$ at $(t_n,t_{n+1})$ with jump discontinuities at $t_n$, for $n \geq 1$. Assume also that $y'(t)\leq - \lambda y(t)$ on every $(t_n, t_{n+k})$ interval, for some $\lambda >0$. Also, for the discontinuities of $y(t)$ at $t_n$ we assume the bound
\begin{equation*} |\Delta_j y(t_n)|\leq g(n), \quad n \geq 1 ,\end{equation*}
for some function $g(n)$ defined on $\mathbb{N}$. Then we have that \begin{equation} \label{eq:y(t_n)Bound} y(t_n) \leq y_0 e^{-\lambda (t_n -t_0)}+\sum_{k=1}^{n}g(k)e^{-\lambda(t_n -t_k)}, \qquad \text{for} \quad n\geq 1 ,\end{equation}
and for any $t \in [t_n,t_{n+1})$ we have the bound
\begin{equation*} y(t)\leq y_0 e^{-\lambda (t-t_0)}+\sum_{k=1}^{n}g(k)e^{-\lambda(t -t_k)}.\end{equation*}
Moreover, it follows directly that if $\sum_{k=1}^{n}g(k)e^{-\lambda(t_n -t_k)}\xrightarrow{n \to \infty}0$, then $y(t)\xrightarrow{t \to \infty} 0$.\end{lemma}

\begin{proof} The proof is once again a simple exercise on induction. Indeed, for $n=1$ we have that \begin{equation*}y(t_1)=\Delta_j y(t_1)+y(t_1^-) \leq g(1)+ y(t_1^-) \leq g(1) + y_0 e^{-\lambda(t_1-t_0)} . \end{equation*} So \eqref{eq:y(t_n)Bound} holds trivially for $n=1$. Assume that \eqref{eq:y(t_n)Bound} holds for $n-1$, then
\begin{align*} y(t_n)&=\Delta_j y(t_n) + y(t_n^-) \leq g(n)+y(t_{n-1})e^{-\lambda (t_n -t_{n-1})} \\ & \leq g(n)+\left( y_0 e^{-\lambda (t_{n-1} -t_0)}+\sum_{k=1}^{n-1}g(k)e^{-\lambda(t_{n-1} -t_k)} \right)e^{-\lambda (t_n -t_{n-1})} \\ &=y_0 e^{-\lambda (t_n -t_0)}+\sum_{k=1}^{n}g(k)e^{-\lambda(t_n -t_k)}. \end{align*}\end{proof}

\begin{remark} A similar result as in Lemma \ref{lem:Main_lemma} can be concluded for a piecewise $C^1$ function $y(t)$ that satisfies $y'(t) \geq -\lambda y(t)$ (for some $\lambda>0$) and with discontinuity jumps that are bounded below as in $\Delta_j y(t_n) \geq g(n)$ for $n \geq 1$. In this case we can adopt the proof of Lemma \ref{lem:Main_lemma} to show that the reverse inequality of \eqref{eq:y(t_n)Bound} holds. \end{remark}

\begin{proof}[Proof of Theorem 1] The proof of the first sufficient condition \eqref{eq:condition1} is a direct application of Lemma \ref{lem:Main_lemma} and Proposition \ref{thm:prop2}. Consider $y(t)=\mathbb{E}(\pazocal{V}(\{x(t)\}))$ and use Lemma \ref{lem:Main_lemma} with $\lambda=\psi_*$ and with the help of \eqref{eq:Upper_Lower_bound} we may choose $g(n)= \frac{1}{n}(\sigma^2 +\frac{C}{N_0^2})$ so the result follows.

In order to show the second part, we assume that \eqref{eq:condition2} holds and we also have average mean square convergence. Since we assumed $y(t_n) \xrightarrow{n \to \infty} 0$, we can choose some $\epsilon>0$ and  $N(\epsilon)$ s.t. $y(t_n)<\epsilon$, and $c_n >\frac{1}{2}$ for all $n \geq N(\epsilon)$ and also $\frac{1}{2}\sigma^2-\epsilon >0$ (since $\sigma>0$). We know that $y'(t) \geq -\psi_M y(t)$ on $(t_n,t_{n+1})$ and we also have that $\Delta_j y(t_k)\geq g(n)= \frac{1}{2n}\left( \frac{1}{2}\sigma^2 -\epsilon \right)$ for $n \geq N(\epsilon)$. Condition \eqref{eq:condition2} and Lemma \ref{lem:Main_lemma} implies that we don't have convergence of $y(t_n)$ to $0$ which is inconsistent with our initial assumption.

\end{proof}

\begin{proof}[Proof of Theorem 2]

By the definition of $t_j$ we have that $N(t_j)=N_0+j$. Also, $e^{{t_j}^\alpha} -1 \leq N(t_j)= \lfloor e^{{t_j}^\alpha} \rfloor \leq e^{{t_j}^\alpha}$.Thus, we get that $(\ln(j+N_0))^{\frac{1}{\alpha}} \leq t_j \leq (\ln(j+N_0+1))^{\frac{1}{\alpha}}$. For simplicity and from now on we use $p=\frac{1}{\alpha}$. It is evident that for large $t_j \gg 1$ we have that $t_j \sim(\ln j)^p$. We need to check how $e^{-\lambda(\ln n)^p}\sum_{k=1}^n \frac{1}{k}e^{\lambda (\ln k)^p}$ behaves as $n \to \infty$, for the various values of $\lambda>0$, i.e. check if one of conditions $(C1)$ or $(C2)$ holds. It will become evident that no matter the value of $\lambda$ either one of the two conditions holds, hence offering a sharp condition for every increase rate.

We start with the case $p=1$ (exponential growth!) and we have $e^{-\lambda(\ln n)^p}\sum \limits_{k=1}^n \frac{1}{k}e^{\lambda(\ln k)^p}=\frac{1}{n^\lambda}\sum \limits_{k=1}^n k^{\lambda-1}$. For $\lambda=1$ it is $\frac{1}{n}\sum \limits_{k=1}^n 1=1$. For $\lambda>1$, $f(t)=t^{\lambda-1}$ is inreasing and $\frac{1}{n^\lambda}\sum \limits_{k=1}^n k^{\lambda-1}>\frac{1}{n^\lambda}\int_0^n t^{\lambda-1}\, dt =\frac{1}{\lambda}$. Similarly, if $0<\lambda<1$, then $f(t)$ is decreasing and $\frac{1}{n^\lambda}\sum \limits_{k=1}^n k^{\lambda-1} > \frac{1}{n^\lambda}\int_0^n (t+1)^{\lambda-1} \, dt \xrightarrow{n \to \infty} \frac{1}{\lambda}$. As a result, in the case of exponential growth $(C2)$ is satisfied always.

To check the behavior for all other $p$ values we bound the sum (above or below) with the appropriate integral. For a start, note that when $p>1$ the function $f(t)=\frac{1}{t}e^{\lambda(\ln t)^p}$ is eventually increasing (for large enough $t$), and eventually decreasing when $p<1$. Thus, for $p>1$, we have that
\begin{equation*} e^{-\lambda(\ln n)^p}\int_{n_0-1}^n f(t+1) \, dt \geq e^{-\lambda(\ln n)^p} \sum_{k=n_0}^n f(k) , \end{equation*}
whereas for $p<1$ then \begin{equation*}  e^{-\lambda(\ln n)^p} \sum_{k=n_0}^n f(k) \geq e^{-\lambda(\ln n)^p}\int_{n_0-1}^n f(t+1) \, dt .\end{equation*} We compute the integral by substitung $u=\ln (t+1)$ i.e. \begin{equation*}  \int_{n_0-1}^{n} f(t+1) \, dt =\int_{\ln n_0}^{\ln (n+1)} e^{\lambda u^p} \, du .\end{equation*}
We mention here that the function $F(p,x)=e^{-\lambda x^p}\int_0^x e^{\lambda t^p} \, dt$ is the generalized Dawson integral function (see \cite{Dij}), for $p>0$. We can easily check with the help of a bit of elementary analysis (L' H\^opital's rule) that $F(p,x) \xrightarrow{x \to \infty} 0$ iff $p>1$. If we consider  $F(p,x)$ for $x=\ln n$ and $p=\frac{1}{\alpha}$ then it is implied that $e^{-\lambda(\ln n)^p}\sum_{k=1}^n \frac{1}{k}e^{\lambda(\ln k)^p} \xrightarrow{n \to \infty}0$ iff $p>1$.

\textbf{Convergence rate}: In terms of the convergence rate as $t \to \infty$, it is enough to show threre exists some large enough $N(\lambda,p,\beta^*)$ s.t. for $\beta^*<p-1$ \begin{equation*} e^{-\lambda(\ln n)^p}\sum_{k=1}^n \frac{1}{k}e^{\lambda(\ln k)^p} < \frac{1}{(\ln n)^{\beta^*}} \qquad \text{for} \quad n \geq N(\lambda,p,\beta^*).\end{equation*} This is proven by showing that $(\ln n)^{\beta^*}e^{-\lambda(\ln n)^p}\sum_{k=1}^n \frac{1}{k}e^{\lambda(\ln k)^p} \xrightarrow{n \to \infty}0$ for $\beta^*<p-1$. Since $t_n^{\alpha \beta^*}=(\ln n)^{\beta^*}$ and $\alpha \beta^* < 1-\alpha$ the result follows.
\end{proof}

E-MAIL: ioamarkou@iacm.forth.gr

\end{document}